\documentclass[12pt]{article}
\usepackage[margin =1in]{geometry}
\usepackage{amssymb,amsthm}
\usepackage{amsmath}
\usepackage{enumerate}
\usepackage{hyperref}
\usepackage{cite}
\usepackage{color}

\usepackage{mathtools}

\numberwithin{equation}{section}

\newcommand{\norm}[1] {\left \| #1 \right \|}
\newcommand{\inclu}[0] {\ar@{^{(}->}}

\newcommand{\X}{{\mathcal{X}}}
\newcommand{\dist}{{\rm dist}}
\newcommand{\proj}{{\rm proj}}
\newcommand{\R}{\mathbb{R}}

\newcommand{\trace}{\mathrm{Tr}}

\newcommand{\RR}{\mathbb{R}}

\newcommand{\argmin}{\operatornamewithlimits{argmin}}

\newtheorem{thm}{Theorem}[section]

\newtheorem{lem}[thm]{Lemma}

\newtheorem{assumption}{Assumption}

\newtheorem{example}{Example}[section]

\theoremstyle{remark}

\usepackage{mathtools}
\DeclarePairedDelimiter{\dotp}{\langle}{\rangle}
\usepackage[boxruled]{algorithm2e}

\begin{document}
	
	\title{Subgradient methods for sharp weakly convex functions}
	
	
	\author{Damek Davis\thanks{School of Operations Research and Information Engineering, Cornell University,
			Ithaca, NY 14850, USA;
			\texttt{people.orie.cornell.edu/dsd95/}.}
		\and 
		Dmitriy Drusvyatskiy\thanks{Department of Mathematics, U. Washington, 
			Seattle, WA 98195; \texttt{www.math.washington.edu/{\raise.17ex\hbox{$\scriptstyle\sim$}}ddrusv}. Research of Drusvyatskiy was supported by the AFOSR YIP award FA9550-15-1-0237 and by the NSF DMS   1651851 and CCF 1740551 awards.}
		\and
		Kellie J. MacPhee\thanks{Department of Mathematics, U. Washington, 
			Seattle, WA 98195; \texttt{   sites.math.washington.edu/{\raise.17ex\hbox{$\scriptstyle\sim$}}kmacphee}} 
		\and Courtney Paquette\thanks{Industrial and Systems
                  Engineering Department, Lehigh University,
                  Bethlehem, PA 18015;
                  \texttt{sites.math.washington.edu/{\raise.17ex\hbox{$\scriptstyle\sim$}}yumiko88/}. Research
                  of Paquette was supported by NSF CCF 1740796.
                }   
	}

	\date{}
	\maketitle
	
\begin{abstract}
	Subgradient methods converge linearly on a convex function that grows sharply away from its solution set. In this work, we show that the same  is true for sharp functions that are only weakly convex, provided that the subgradient methods are initialized within a fixed tube around the solution set. A variety of statistical and signal processing tasks come equipped with good initialization, and provably lead to formulations that are both weakly convex and sharp. Therefore,  in such settings, subgradient methods can serve as inexpensive local search procedures. We illustrate the proposed techniques on phase retrieval and covariance estimation problems.
	\end{abstract}

	\section{Introduction}
	Typical methods for statistics and signal processing tasks follow the two-step strategy: $(1)$ find a moderately accurate solution $\hat x$ at a low sample complexity cost (e.g., using spectral initialization), and $(2)$ refine $\hat x$ by an iterative ``local search algorithm'' that converges rapidly under natural statistical assumptions. For smooth problem formulations, the term ``local search'' almost universally refers to gradient descent or a close variant thereof; see e.g. \cite{pmlr-v40-Jain15,Tu:2016:LSL:3045390.3045493,wirt_flow,JJKN17,chen2015fast,meka2009guaranteed,boumal2016nonconvex,brutzkus2017globally}. For nonsmooth and nonconvex problems, the meaning of local search is much less clear.
	In this work, we ask the following question.
	
	\begin{quote}
	Is there a generic gradient-based local search procedure for nonsmooth and nonconvex problems, which converges linearly under standard regularity conditions?
	\end{quote}
	
	
	Not surprisingly, our approach is rooted in  subgradient methods for convex optimization. To motivate the discussion,	consider the constrained optimization problem
	\begin{equation}\label{eqn:target_prob}
	\min_{x\in \mathcal{X}}~ g(x),
	\end{equation}
	where $g$ is an $L$-Lipschitz convex function on $\R^d$ and $\mathcal{X}$ is a closed convex set. Given a current iterate $x_k$, subgradient methods proceed as follows: 
	\begin{equation*}
	\left\{
	\begin{aligned}
	&\textrm{Choose any }\zeta_k\in \partial g(x_k)\\
	&\textrm{Set } x_{k+1}=\proj_{\mathcal{X}}\left(x_k-\alpha_k\cdot\frac{\zeta_k}{\|\zeta_k\|}\right)
	\end{aligned}\right\}.
\end{equation*}
	Here, the symbol $\proj_{\mathcal{X}}(y)$ denotes the nearest point of  $\mathcal{X}$ to $y$ and $\{\alpha_k\}$ is a specified stepsize sequence. The choice of the sequence $\{\alpha_k\}$ determines the behavior of the scheme, and is the main distinguishing feature among  subgradient methods. In this work, we will only be interested in subgradient methods that are linearly convergent. As usual, linear rates of convergence of iterative methods require some regularity conditions to hold. Here, the appropriate regularity condition is {\em sharpness} \cite{pol_sharp,weak_sharp} (or equivalently a global error bound): there exists a real $\mu>0$ satisfying 
	$$g(x)-\min_{x\in \mathcal{X}} g\geq \mu\cdot \dist(x;\X^*)\qquad \textrm{for all } x\in \X,$$
	where $\X^*$ denotes the set of minimizers of \eqref{eqn:target_prob}. 
	Assuming sharpness holds,  subgradient methods, with a judicious choice of $\{\alpha_k\}$, produce iterates that converge to $\X^*$ at the linear rate $\sqrt{1-(\mu/L)^2}$. Results of this type date back to 60's and 70's \cite{goff,erem_subgrad,min_unsmooth,shor_rate_subgrad,subgrad_surv}, while some more recent approaches have appeared in \cite{stagg,rsg,jstone}.

	Various contemporary problems lead to formulations that are indeed sharp, but are only {\em weakly convex} and {\em locally Lipschitz}. Recall that a function $g$ is $\rho$-weakly convex \cite{Nurminskii1973} if the perturbed function $x\mapsto g(x)+\frac{\rho}{2}\|\cdot\|^2$ is convex for some $\rho>0$. Note that weakly convex functions need not be smooth nor convex.  A quick computation (Lemma~\ref{lem:region_without_stat}) shows that if $g$ is $\mu$-sharp and $\rho$-weakly convex, then there is a tube around the solution set $\X^*$ that contains no extraneous stationary points:
	$$\mathcal{T}:=\left\{x\in\X: \dist(x;\X^*)\leq \frac{2\mu}{\rho}\right\}.$$
	In this work, we show that the standard linearly convergent subgradient methods originally designed for convex problems, apply in this much greater generality, provided they are initialized within a slight contraction of the tube $\mathcal{T}$. The methods exhibit essentially the same linear rate of convergence as in the convex case, while the weak convexity constant $\rho$ only determines the validity of the initialization. We focus on three step-size rules: Polyak stepsize \cite{erem_subgrad,min_unsmooth}, geometrically decaying step \cite{shor_rate_subgrad,goff}, and constant stepsize \cite{stagg,rsg,jstone}. As proof of concept, we illustrate the resulting algorithms on phase retrieval and covariance estimation problems. 
	
	Our current work sits within the broader scope of analyzing subgradient and proximal methods for weakly convex problems \cite{Nurminskii1973,Nurminskii1974,davis2018stochastic,prox_error,duchi_ruan,duchi_ruan_PR,comp_DP,prixm_guide_subgrad}; see also the recent survey \cite{prox_point_surv}. In particular, the  paper \cite{davis2018stochastic}  proves a global sublinear rate of convergence, in terms of a natural stationarity measure, of a (stochastic) subgradient method on any weakly convex function. In contrast, here we are interested in subgradient methods that are locally linearly convergent under the additional sharpness assumption. The arguments we present are all quick modification of the proofs already available in the convex setting. Nonetheless, we believe that the drawn conclusions are interesting and  powerful,  opening the door to generic local search procedures for nonsmooth and nonconvex problems.

	

	\section{Notation}
	Throughout, we consider the Euclidean space $\R^d$, equipped with the inner-product $\langle \cdot,\cdot\rangle$ and the induced norm $\|x\|:=\sqrt{\langle x,x\rangle}$. The {\em distance} and the {\em projection} of any point $y\in\R^d$ onto a set $\mathcal{X}$, are defined by 
\begin{align*}
\dist(y;\mathcal{X}):=\inf_{x\in \mathcal{X}}~\|y-x\|\qquad\qquad\textrm{and}\qquad\qquad
\proj_{\mathcal{X}}(y):=\argmin_{x\in \mathcal{X}}~\|y-x\|,
\end{align*}
respectively. Note that $\proj_{\mathcal{X}}(y)$ is nonempty as long as $\mathcal{X}$ is a closed set. The {\em indicator function} of a set $\X$, denoted by $\delta_{\X}$, is defined to be zero on $\X$ and $+\infty$ off it.

\subsection{Weakly convex functions}
Our main focus is on those functions that are convex up to  an additive quadratic  perturbation.
 Namely, a function $g\colon\R^d\to\R\cup\{+\infty\}$ is called $\rho$-{\em weakly convex} (with $\rho\geq 0$) if the assignment $x\mapsto g(x)+\tfrac{\rho}{2}\|x\|^2$ is a convex function. The algorithms we consider will all use generalized derivative constructions.
 Variational analytic literature highlights a number of distinct subdifferentials (e.g. \cite{RW98,Mord_1,penot_book}); for weakly convex functions, all these constructions coincide.
  Consider a $\rho$-weakly convex function $g$. 
The {\em subdifferential}  of $g$ at $ x$, denoted  ${\partial} g( x)$, is the set  of all vectors $v\in\R^d$ satisfying
\begin{equation}\label{eqn:weak_subdiff}
	g(y)\geq g( x)+\langle v,y- x\rangle +o(\|y- x\|)\qquad \textrm{ as }y\to  x.	
\end{equation}
Though the condition \eqref{eqn:weak_subdiff} appears to lack uniformity with respect to the basepoint $x$, the subgradients of $g$ automatically satisfy the much stronger property \cite[Theorem 12.17]{RW98}:
\begin{equation}\label{eqn:stronger_ineq}
g(y)\geq g(x)+\langle v,y-x\rangle-\frac{\rho}{2}\|y-x\|^2,\qquad\forall x,y\in \R^d,~v\in \partial g(x).
\end{equation}
 Thus we may use the two conditions, \eqref{eqn:weak_subdiff} and \eqref{eqn:stronger_ineq}, interchangeably for weakly convex functions. We note in passing that localizing condition \eqref{eqn:stronger_ineq} leads to so-called prox-regular functions, introduced in \cite{prox_reg}.

 Weakly convex functions are widespread in applications and are typically easy to recognize. One common source is the composite problem class:
\begin{equation}\label{eqn:comp}
\min_{x}~ F(x):=h(c(x)),
\end{equation}
where $h\colon\R^m\to\R$ is convex and $L$-Lipschitz, and $c\colon\R^d\to\R^m$ is a $C^1$-smooth map with $\beta$-Lipschitz gradient. An easy argument shows that  $F$ is $L\beta$-weakly convex. This is a worst case estimate. In concrete circumstances, the composite function $F$ may have a much more favorable weak convexity constant $\rho$. The elements of the subdifferential $\partial F(x)$ are straightforward to compute through the chain rule \cite[Theorem 10.6, Corollary 10.9]{RW98}:
	$$\partial F(x)=\nabla c(x)^{*}\partial h(c(x))\qquad \textrm{for all }x\in \R^d.$$
	For a discussion of some recent uses of weakly convex functions in optimization, see the short survey \cite{prox_point_surv}.
	Throughout the paper, we will use the following two running examples to illustrate our results.
	
\begin{example}[Phase retrieval]
	{\rm
		Phase retrieval is a common computational problem, with applications in 
		diverse areas such as imaging, X-ray crystallography, and speech processing. 
		For simplicity, we will focus on the version of the problem over the reals.
		The (real) phase retrieval problem seeks to determine a point $x$ satisfying the 
		magnitude conditions, $$|\langle a_i,x\rangle|^2\approx b_i\quad \textrm{for 
		}i=1,\ldots,m,$$ where $a_i\in \R^d$ and $b_i\in\R$ are given. 
		Note that we can only recover the optimal $x$ up to a universal sign change, since $|\langle a_i,x\rangle| = |\langle a_i,-x\rangle|$.
		In this work, we will focus on the following optimization formulation of the problem \cite{eM,duchi_ruan_PR, proj_weak_dim}:
		$$\min_x ~\frac{1}{m}\sum_{i=1}^m |\langle a_i,x\rangle^2-b_i|.$$
		Clearly, this is an instance of \eqref{eqn:comp}.
		Indeed, under mild statistical assumptions on the way $a_i$ are generated, the  formulation is $\rho$-weakly convex, for some numerical constant $\rho$ independent of $d$ and $m$ \cite[Corollary 3.2]{duchi_ruan_PR}.  Moreover, under an appropriate model of the noise in the measurements, the problem is sharp \cite[Propostion 3]{duchi_ruan_PR}.
		It is worthwhile to mention that numerous 
		other approaches to phase retrieval exist, based on different 
		problem 
		formulations; see for example \cite{wirt_flow,rand_quad,phase_nonconv,versh_kac}.
			
		Experiment set-up: All of the experiments on phase retrieval will be generated according to the following procedure. In the {\em exact set-up}, we  generate standard Gaussian measurements $a_i\sim N(0,I_{d\times d})$,  for $i=1,\ldots, m$, and generate the target signal $\bar x\sim N(0,I_{d\times d})$. We then set $b_i = \langle a_i, \bar x\rangle^2$ for each $i=1,\ldots, m$. In the {\em corrupted set-up}, we generate $a_i$ and $\bar x$ as in the noiseless case. We then corrupt a proportion of the measurements with  outliers. Namely, we set $b_i=(1-z_i) \langle a_i,\bar x\rangle^2 +z_i |\zeta_i|$, where $z_i \sim \textrm{Bernoulli}(0.1)$ and $\zeta_i\sim \mathcal{N}(0,100)$. 		}		
	\end{example}

\begin{example}[Covariance matrix estimation]{\rm
	The problem of covariance estimation from quadratic measurements, introduced in  \cite{chen2015exact}, is a higher rank variant of  phase retrieval. Let $a_1, \ldots, a_m \in\R^d$ be measurement vectors. The goal is to recover a low rank decomposition of a covariance matrix $\bar X\bar X^T$, with $\bar X\in\RR^{d \times r}$, from quadratic measurements
	$$
	b_i  \approx a_i^T \bar X\bar X^T a_i  = \trace(\bar X\bar X^T a_ia_i^T).
	$$
	Note that we can only recover $\bar X$ up to multiplication by an orthogonal matrix.
	This problem arises in a variety of contexts, such as covariance sketching for data streams and spectrum estimation of stochastic processes. We refer the reader to \cite{chen2015exact} for details.
	In our examples, we will assume $m$ is even and will focus on the potential function 
	\begin{equation}\label{eqn:cov_est}
			\min_x ~\frac{1}{m}\sum_{i=1}^m \left|\left\langle XX^T,a_{2i}a_{2i}^T-a_{2i-1}a_{2i-1}^T\right\rangle-(b_{2i}-b_{2i-1})\right|.
			\end{equation}
	Under exact measurements, i.e., $b_i  = a_i^T \bar X\bar X^T a_i $ and under appropriate statistical assumptions on how $a_i$ are generated,  the formulation \eqref{eqn:cov_est} is  $\rho$-weakly convex for a numerical constant $\rho$, independent of $d$ or $m$, and is sharp. Indeed, it is a simple consequence of two results, namely~\cite[Corollary 1]{chen2015exact} and~\cite[Lemma 5.4]{Tu:2016:LSL:3045390.3045493}. It is possible to show the objective is also sharp when the measurements are corrupted by gross outliers. This guarantee is beyond the scope of our current work, and will appear in a different paper.

		Experiment set-up: All of the experiments on covariance matrix estimation will be generated according to the following procedure. In the {\em exact set-up}, we  generate standard Gaussian measurements $a_i\sim N(0,I_{d\times d})$   for $i=1,\ldots, m$, and generate the target matrix $\bar X\in\R^{d\times r}$ as a standard Gaussian. We then  set $b_i = \|{\bar X}^T a_i\|^2_F$ for each $i=1,\ldots, m$. In the {\em corrupted set-up}, we generate $a_i$ and $\bar X$ as in the exact case. We then corrupt a proportion of the measurements with  outliers. Namely, we set $b_i=(1-z_i)\| \bar X^Ta_i \|_F^2+z_i |\zeta_i|$, where $\zeta_i\sim \mathcal{N}(0,100)$ and $z_i \sim \textrm{Bernoulli}(0.1)$. All plots will show iteration counter $k$ versus the scaled Procrustes distance $\displaystyle\dist(X_k, \mathcal{X}^*) / \| \bar X\| = \min_{\Omega^T\Omega = I} \|\Omega X_k-\bar X\|_F/\|\bar X\|_F$.

}
\end{example}

\subsection{Setting of the paper}

	Throughout the manuscript, we make the following assumption.

	\begin{assumption}\label{assum:weak_conv_sharp}
	{\rm Consider the optimization problem 
		\begin{equation}\label{eqn:stat}
		\min_{x\in\mathcal{X}}~ g(x),
		\end{equation}
satisfying the following properties for some real $\mu,\rho >0$.
	\begin{enumerate}
	\item {\bf (Weak-convexity)} The function $g\colon\R^d\to\R$ is $\rho$-weakly convex, and the set $\mathcal{X}\subset\R^d$ is  closed and convex. The set of minimizers $\displaystyle \mathcal{X}^*:=\argmin_{x\in \mathcal{X}} g(x)$ is nonempty.
	\item {\bf (Sharpness)} The inequality 
	$$g(x)-\min_{ \mathcal{X}} g\geq \mu\cdot \dist(x;\mathcal{X}^*)\qquad \textrm{ holds for all }x\in \mathcal{X}.$$
	\end{enumerate}}
	\end{assumption}

We will say that a point $\bar x\in \mathcal{X}$ is {\em stationary} for the target problem  \eqref{eqn:stat} if
	$$g(x)-g(\bar x)\geq o(\|x-\bar x\|)\qquad \textrm{ as }x\xrightarrow{\mathcal{X}} \bar x.
$$
That is, $\bar x$ is {\em stationary} precisely when the zero vector is a subgradient of $g+\delta_{\mathcal{X}}$ at $\bar x$.

Shortly, we will discuss subgradient methods that converge linearly to $\X^*$ under appropriate initialization. As a first step, therefore, we must identify a neighborhood of $\X^*$ that is devoid of extraneous stationary points of \eqref{eqn:stat}. This is the content of the following lemma.

	\begin{lem}[Neighborhood with no stationary points]\label{lem:region_without_stat}
		The problem \eqref{eqn:stat} has no stationary points $x$ satisfying 
		\begin{equation}\label{eq:region_for_linear}
		0<\dist(x;\X^*)< \frac{2\mu}{\rho}.
		\end{equation}
	\end{lem}
	\begin{proof}
		Fix a stationary point $x\in\X\setminus \X^*$ of \eqref{eqn:stat}. Choosing an arbitrary $\bar x \in\proj_{\X^*}(x)$,
		observe
		\begin{align*}
		\mu\cdot \dist(x; \X^*) \leq g(x)-g(\bar x) \leq  \frac{\rho}{2}\|x - \bar x\|^2=\frac{\rho}{2}\cdot\dist^2(x; \X^*).
		\end{align*}
		Dividing through by $\dist(x; \X^*)$, the result follows.
	\end{proof}

	In light of Lemma~\ref{lem:region_without_stat}, for any
        $\gamma > 0$ define the following tube 
	$$\mathcal{T}_{\gamma}:=\left\{x\in \mathcal{X}:\dist(x;\mathcal{X}^*) < \gamma\cdot \frac{\mu}{\rho}\right\}$$
	and the  constant
		\begin{equation}
		L:= \sup\left\{\|\zeta\|: \zeta\in \partial g(x),\,
		x\in \mathcal{T}_{1}\right\}. \label{eq: constant_L}
		\end{equation}
		Lemma~\ref{lem:region_without_stat} guarantees that the tubes $\mathcal{T}_{\gamma}$ contain no extraneous stationary points of the problem for any $\gamma\in (0,2]$. Moreover, observe that $\mu$ and $L$ play reciprocal roles; consequently, the ratio $\tau:=\mu/L$ should serve as a measure of conditioning. The following  lemma verifies the inclusion $\tau\in [0,1]$.
	
	\begin{lem}[Condition number]
		The inclusion $\tau\in [0,1]$ holds.
	\end{lem}
	\begin{proof}
	Consider an arbitrary point $x\in \mathcal{T}_1\setminus \X^*$ and choose a point  $\bar x\in\proj_{\X^*}(x)$. By Lebourg's mean value theorem \cite[Theorem 2.3.7]{clarke}, there exists a point $z$ in the open segment $(x,\bar x)$ and a vector $\zeta\in \partial g(z)$ satisfying 
	\begin{equation}\label{eqn:mean_val_thm}
	g(x)-g(\bar x)=\langle \zeta,x-\bar x\rangle.
	\end{equation}
	Trivially $z$ lies in $\mathcal{T}_1$, and therefore $\|\zeta\|\leq L$.
	Using this estimate and sharpness in \eqref{eqn:mean_val_thm} yields the guarantee 
	$$\mu\cdot \dist(x;\mathcal{X}^*)\leq g(x)-g(\bar x)\leq \|\zeta\|\cdot\|x-\bar x\|\leq L\cdot\dist(x;\mathcal{X}^*).$$
The result follows.	
\end{proof}

	To summarize, we will use the following symbols to describe the parameters of the problem class \eqref{eqn:stat}: $\rho$ is the weak convexity constant of $g$, $\mu$ is the sharpness constant of $g$, $L$ is the maximal subgradient norm at points in the tube $\mathcal{T}_1$, and $\tau$ is the condition measure $\tau=\mu/L\in [0,1]$.

	\section{Polyak subgradient method}
In this section, we consider the Polyak subgradient method for the problem \eqref{eqn:stat}. A preliminary version of this material applied to the phase retrieval problem appeared in \cite{proj_weak_dim}; we present the arguments here for the sake of completeness.

The Polyak subgradient method is summarized in Algorithm~\ref{alg:polyak}.  This method requires knowing the optimal value $\min_{x\in\X} g(x)$. In a number of circumstances, this indeed is reasonable (e.g. exact penalty approach for solving nonlinear equations). The latter sections explore subgradient methods that do not require a known optimal value.

\bigskip

	\begin{algorithm}[H]
		\KwData{$x_0 \in \RR^d$}
		{\bf Step $k$:} ($k\geq 0$)\\
		$\qquad$ Choose $\zeta_k \in \partial g(x_k)$. {\bf If} $\zeta_k=0$, then exit algorithm.\\
	$\qquad$ Set $\displaystyle x_{k+1}=\proj_{\X}\left(x_{k} - \frac{g(x_k)-\min_{\mathcal{X}} g}{\|\zeta_k\|^2}\zeta_k\right)$.

		\caption{Polyak Subgradient Method}
		\label{alg:polyak}
	\end{algorithm}
	\bigskip
	
	The following theorem shows that Algorithm~\ref{alg:polyak}, originally proposed for convex problems, enjoys the same linear convergence guarantees for functions that are only weakly convex, provided it is initialized within a certain tube of the optimal solution set.

	\begin{thm}[Linear rate]\label{thm:qlinear}
	Fix a real $\gamma \in (0,1)$. 
		Then  Algorithm~\ref{alg:polyak} initialized at any point $x_0\in \mathcal{T}_{\gamma}$ produces iterates that 
		converge $Q$-linearly to $\X^*$, that is 
		\begin{align}\label{eqn:lin_rate1}
		  \dist^2(x_{k+1};\X^*) \leq \left(1-(1-\gamma) \tau^2\right)\dist^2(x_{k};\X^*).
		\end{align}
	\end{thm}
	\begin{proof}
		We proceed by induction. Suppose that the theorem holds up to iteration $k$.  We will prove the inequality \eqref{eqn:lin_rate1}.
		To this end, choose $\bar x \in \proj_{\X^*}(x_k)$. Note that if $x_k$ lies in $\X^*$, there is nothing to prove. Thus we may suppose $x_k\notin \X^*$.
		Note that the inductive hypothesis implies $\dist(x_{k};\X^*)\leq \dist(x_{0};\X^*)$ and therefore $x_k$ lies in $\mathcal{T}_{\gamma}$. Lemma~\ref{lem:region_without_stat} therefore guarantees
		$\zeta_k\neq 0$.
		We successively deduce, by non-expansiveness of $\proj_{\mathcal{X}}$, that 
		\begin{align*}
		\|x_{k+1} - \bar x\|^2&\leq\left\|(x_{k} -\bar x)- \tfrac{g(x_k)-\min_{\mathcal{X}} g}{\|\zeta_k\|^2} \zeta_k\right\|^2 \\
		&= \|x_{k} - \bar x\|^2 + \frac{2(g(x_k) - g(\bar x))}{\|\zeta_k\|^2}\cdot\dotp{\zeta_k,\bar x - x_{k}} + \frac{(g(x_k) - g(\bar x))^2}{\|\zeta_k\|^2}\\
		&\leq \|x_{k} - \bar x\|^2 + \frac{2(g(x_k) - g(\bar x))}{\|\zeta_k\|^2}\left( g(\bar x) - g(x_k) + \frac{\rho}{2}\|x_k - \bar x\|^2 \right)\ + \frac{(g(x_k) - g(\bar x))^2}{\|\zeta_k\|^2} \\
		&= \|x_{k} - \bar x\|^2 + \frac{(g(x_k) - g(\bar x))}{\|\zeta_k\|^2}\left(\rho\|x_k - \bar x\|^2 -  (g(x_k) - g(\bar x))  \right)\\
		&\leq \|x_{k} - \bar x\|^2 + \frac{(g(x_k) - g(\bar x))}{\|\zeta_k\|^2}\left(\rho\|x_k - \bar x\|^2 -  \mu\|x_k - \bar x\|  \right) \\
		&= \|x_{k} - \bar x\|^2 + \frac{\rho (g(x_k) - g(\bar x))}{\|\zeta_k\|^2}\left( \|x_k - \bar x\| -\frac{\mu}{\rho}\right)\|x_k - \bar x\|.
		\end{align*}
		Combining the inclusion $x_k\in \mathcal{T}_{\gamma}$ with sharpness, we therefore deduce
		$$\dist^2(x_{k+1};\X^*)\leq \|x_{k+1} - \bar x\|^2\leq \left(1-\frac{(1-\gamma) \mu^2}{\|\zeta_k\|^2}\right)\|x_k-\bar x\|^2.$$
		The result follows.
	\end{proof}

	As a numerical illustration, let us apply the Polyak subgradient method (Figure~\ref{fig:pol_phase_cov}) to our two running examples, phase retrieval and covariance matrix estimation.  Notice that a linear rate of convergence is observed in all experiments except for two, with the rate improving monotonically  with an increasing number of measurements $m$. In the two exceptional experiments, the number of measurements $m$ is too small to guarantee that the initial point $x_0$ is within the basin of attraction, and the subgradient methods stagnates.
		
		\begin{figure}[!h]
		\centering
			\includegraphics[scale=0.55]{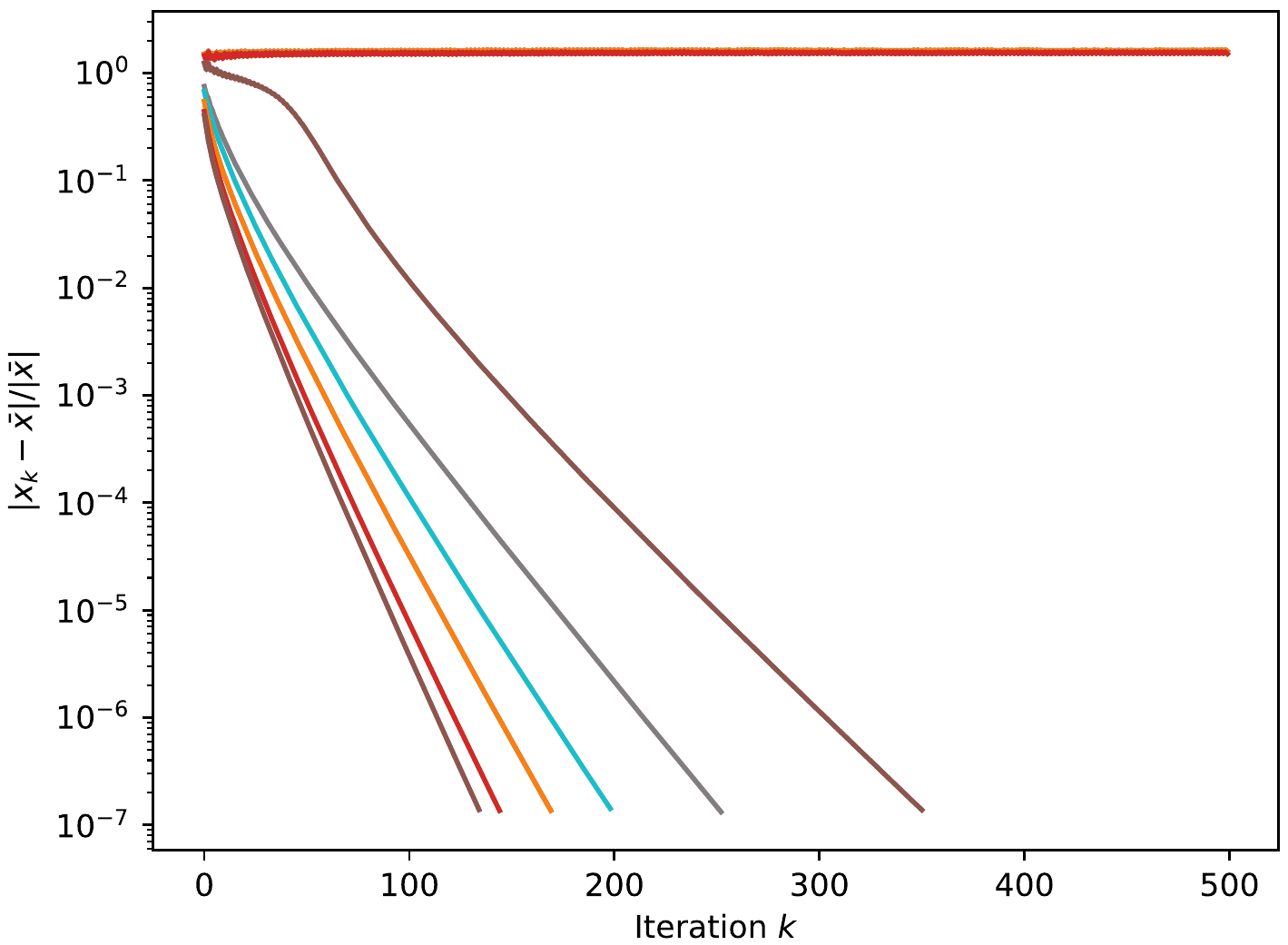}
			\includegraphics[scale=0.55]{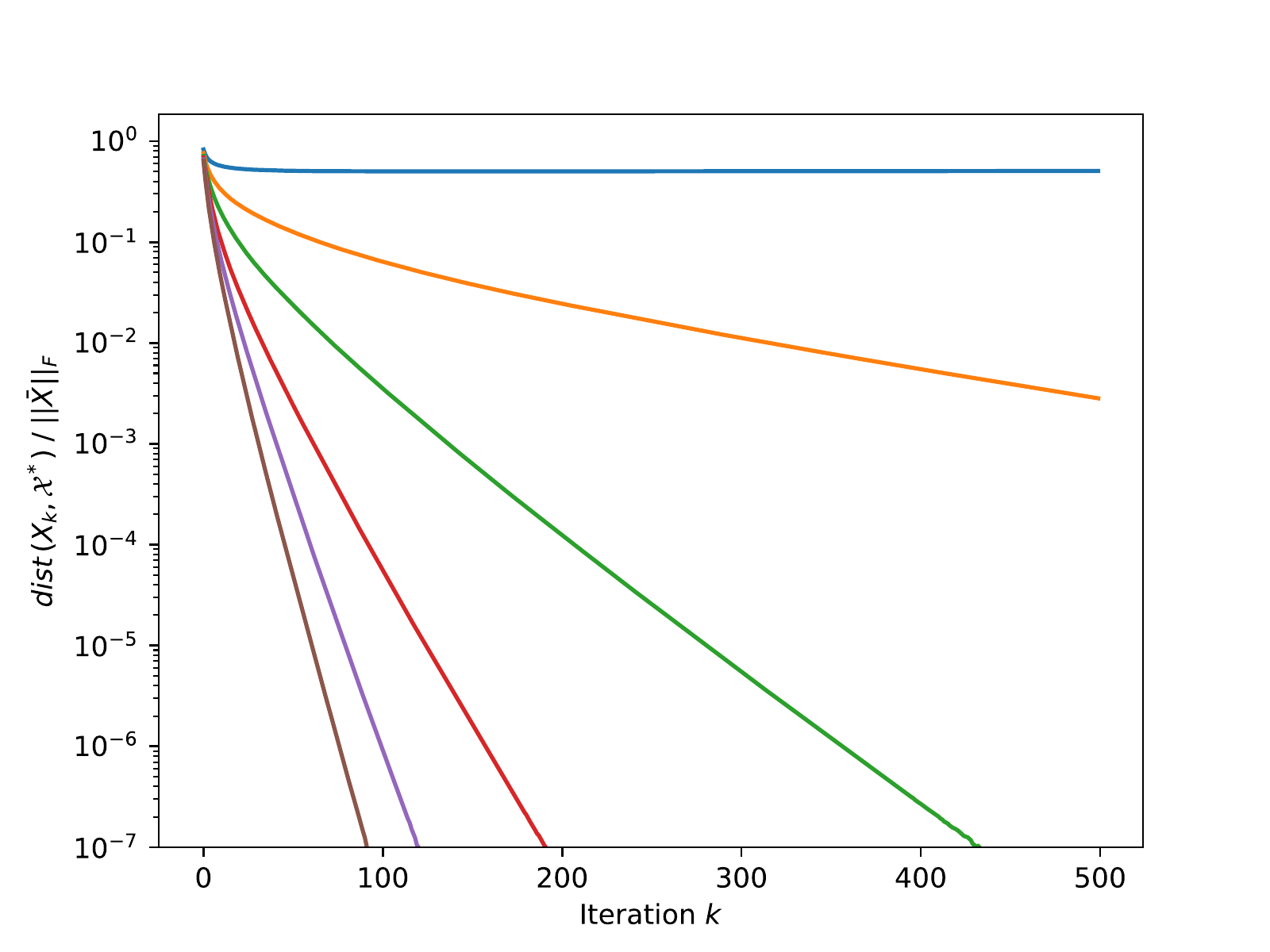}
			\caption{Polyak subgradient method. (Left) { Phase retrieval}  with the exact set-up; $d=5000$ and $m\in \{11000,12225, 13500, 14750, 16000, 17250,18500\}$. (Right) { Covariance matrix estimation} with the exact set-up;  $d=1000$, $r=3$, and $m\in \{5000, 8000, 11000, 14000, 17000, 20000\}$. In both experiments, convergence rates  uniformly improve with increasing $m$.
			}
			\label{fig:pol_phase_cov}
		\end{figure}

	\section{Subgradient method with constant step-size}
	Recall that the Polyak subgradient method
        (Algorithm~\ref{alg:polyak}) crucially relies on knowing the
        minimal value of the optimization problem
        \eqref{eqn:stat}. Henceforth, all the subgradient methods we
        consider are agnostic to this value. That being said, they
        will require some estimates on  the problem parameters $(\mu,
        \rho, L)$. We begin by analyzing a subgradient method with a
        constant step size
        (Algorithm~\ref{alg:const_step}). Constant-step schemes are
        often methods of choice in practice. We will show that when
        properly initialized, the subgradient method with constant
        stepsize generates iterates $x_k$ such that $\dist(x_k;\mathcal{X}^*)$ converges linearly up to a certain threshold.

	\smallskip
	
	\begin{algorithm}[H]
		\KwData{Initial point $x_0 \in \RR^d$ and stepsize $\alpha>0$}
		{\bf Step $k$:} ($k\geq 0$)\\
		$\qquad$ Choose $\zeta_k \in \partial g(x_k)$. {\bf If} $\zeta_k=0$, then exit algorithm.\\
$\qquad$ Set $x_{k+1}=\proj_{\X}\left(x_{k} - \alpha\cdot\frac{\zeta_k}{\|\zeta_k\|}\right)$.


		
		\caption{Subgradient method with constant stepsize}
		\label{alg:const_step}
	\end{algorithm}
	\bigskip

	The analysis we present fundamentally relies on the following estimate, often used in the analysis of subgradient methods. To simplify notation, for any point $x\in\R^d$, we set 
	$$E(x):=\dist^2(x;\X^*).$$
	Whenever $x$ has an index $k$ as a subscript, we will set $E_k:=E(x_k)$. The following lemma will feature in both the constant and geometrically decaying stepsize schemes.

	\begin{lem}[Basic recurrence]\label{lem:key_rec}
		Consider a point $x\in\mathcal{T}_1$ and a nonzero subgradient $\zeta\in \partial g(x)$, and define $x^+:=\proj_{\X}\left(x-\alpha\frac{\zeta}{\|\zeta\|}\right)$ for some $\alpha>0$. Then the estimate holds:		\begin{equation}\label{eqn:claimed_est}
		E(x^+)\leq \left(1+\tfrac{\rho\alpha}{L}\right)E(x)-2
		\alpha\tau\sqrt{E(x)} +\alpha^2.
		\end{equation}
	\end{lem}
	\begin{proof}
	Choose an arbitrary point $\bar x \in \proj_{\X^*}(x)$. Observe
	\begin{align*}
	\|x^+-\bar x\|^2\leq \left\|(x-\bar x)-\alpha\tfrac{\zeta}{\|\zeta\|}\right\|&=\|x-\bar x\|^2+\tfrac{2\alpha}{\|\zeta\|}\cdot\langle \zeta,\bar x-x
	\rangle +\alpha^2\\
	&\leq \|x-\bar x\|^2 +\tfrac{2\alpha}{\|\zeta\|}\cdot\left(g(\bar x)-g(x)+\frac{\rho}{2}\|x-\bar x\|^2\right)+\alpha^2\\
	&\leq \left(1+\frac{\alpha\rho}{\|\zeta\|}\right)\|x-\bar x\|^2-\frac{2\alpha \mu}{\|\zeta\|}\cdot\|x-\bar x\| +\alpha^2.
	\end{align*}
	Thus the inequality holds:
	\begin{equation*} 
	E(x^+)\leq \left(1+\tfrac{\alpha\rho}{\|\zeta\|}\right)E(x)-\tfrac{2\alpha \mu}{\|\zeta\|}\cdot\sqrt{E(x)} +\alpha^2,
	\end{equation*}
	and consequently taking into account $\alpha/\|\zeta\| \geq \alpha/L$ we have
	$$E(x^+)\leq \sup_{t\geq\alpha/L} \left\{\left(1+\rho t\right)E(x)-2\mu t\cdot\sqrt{E(x)} +\alpha^2\right\}.$$
	Notice, the function inside the supremum is linear in $t$ with slope $s:=\rho E(x)-2\mu\sqrt{E(x)}$. The inclusion $x\in \mathcal{T}_1$ directly implies  $s\leq 0$. Therefore the supremum on the right-hand-side is attained at $t=\tfrac{\alpha}{L}$, yielding the claimed estimate \eqref{eqn:claimed_est}.
	\end{proof}

	In light of Lemma~\ref{lem:key_rec}, we can now prove that the quantities $E(x_k)$ converge linearly below a certain fixed threshold. The proof is a modification of that in \cite[Section 4]{jstone}. 
	
	\begin{lem}[Contraction inequality]\label{lem:contract}
	Fix a constant $\alpha\in (0,\frac{\tau \mu}{\rho})$ and let $\{x_k\}_{k\geq 0}$ be the iterates generated by Algorithm~\ref{alg:const_step}. Define the quantity
		\begin{equation}\label{eqn:estar}
		E^* := \left ( \frac{\alpha L}{\mu +
			\sqrt{\mu^2-\alpha \rho L}}  \right )^2.
		\end{equation}
			Then whenever an iterate $x_k$ lies in $\mathcal{T}_1$, the estimate holds:
$$			E_{k+1}-E^*\leq  q_k(E_k-E^*),
		$$
	where $q_k:=1+\tfrac{\alpha}{L}\left(\rho-\tfrac{2 \mu}{\sqrt{E_k}+\sqrt{E^*}}\right)$ satisfies $q_k< 1$.
	\end{lem}
\begin{proof}
		Looking back at the estimate \eqref{eqn:claimed_est}, consider  the following equation in the variable $e$:
		\begin{equation}\label{eqn:fixed_point1}
		e= \left(1+\tfrac{\alpha\rho}{L}\right)e-2\alpha\tau\cdot\sqrt{e}+\alpha^2. 
		\end{equation} 
		An easy computation shows that the minimal positive
		solution to \eqref{eqn:fixed_point1} is exactly $E^*$, defined in \eqref{eqn:estar}.
		Note that $E^*$ is well-defined by the inequality $\alpha\leq \tau\cdot \frac{\mu}{\rho}$.

	Subtracting \eqref{eqn:fixed_point1} from \eqref{eqn:claimed_est}  yields the estimate
		\begin{align*}
		E_{k+1}-E^*&\leq (1+\tfrac{\alpha\rho}{L})(E_k-E^*)-2\alpha \tau(\sqrt{E_k}-\sqrt{E^*})\\
		&= \left(1+\tfrac{\alpha}{L}\left(\rho-\tfrac{2 \mu}{\sqrt{E_k}+\sqrt{E^*}}\right)\right)(E_k-E^*).
		\end{align*}	
		Finally, notice 
		$$\rho-\frac{2\mu}{\sqrt{E_k}+\sqrt{E^*}}< \rho-\frac{2\mu}{2\mu/\rho}= 0.$$
This completes the proof of the lemma.
\end{proof}

	Iterating Lemma~\ref{lem:contract}, we see that  the quantities $E_k$ decrease to a value lower than $E^*$ at a linear rate. Figure~\ref{fig:const_step} illustrates this behavior on our two running examples. It is also clear from the figure that the linear rate of convergence improves as $E_k$ tends to $E^*$. An explanation is immediate  from the expression for $q_k$ in Lemma~\ref{lem:contract}. Indeed, as $E_k$ decreases, so do the contraction factors $q_k$, and for $E_k\approx E^*$, we have $q_k\approx (1-\tau^2)+\frac{\alpha\rho}{L}$. Thus as the step-size tends to zero, the limiting linear rate coincides with the ideal rate of $1-\tau^2$. 
	
	Another interesting feature, apparent in Figure~\ref{fig:const_step}, is that even after $E_k$ becomes smaller than $E^*$, all the following values $E_k$ stay close to $E^*$. This is the content of the following theorem. The convex version of this theorem appears in \cite[Theorem 2]{jstone}.
	\smallskip
	
			\begin{figure}[!h]
				\centering
				\includegraphics[scale=0.56]{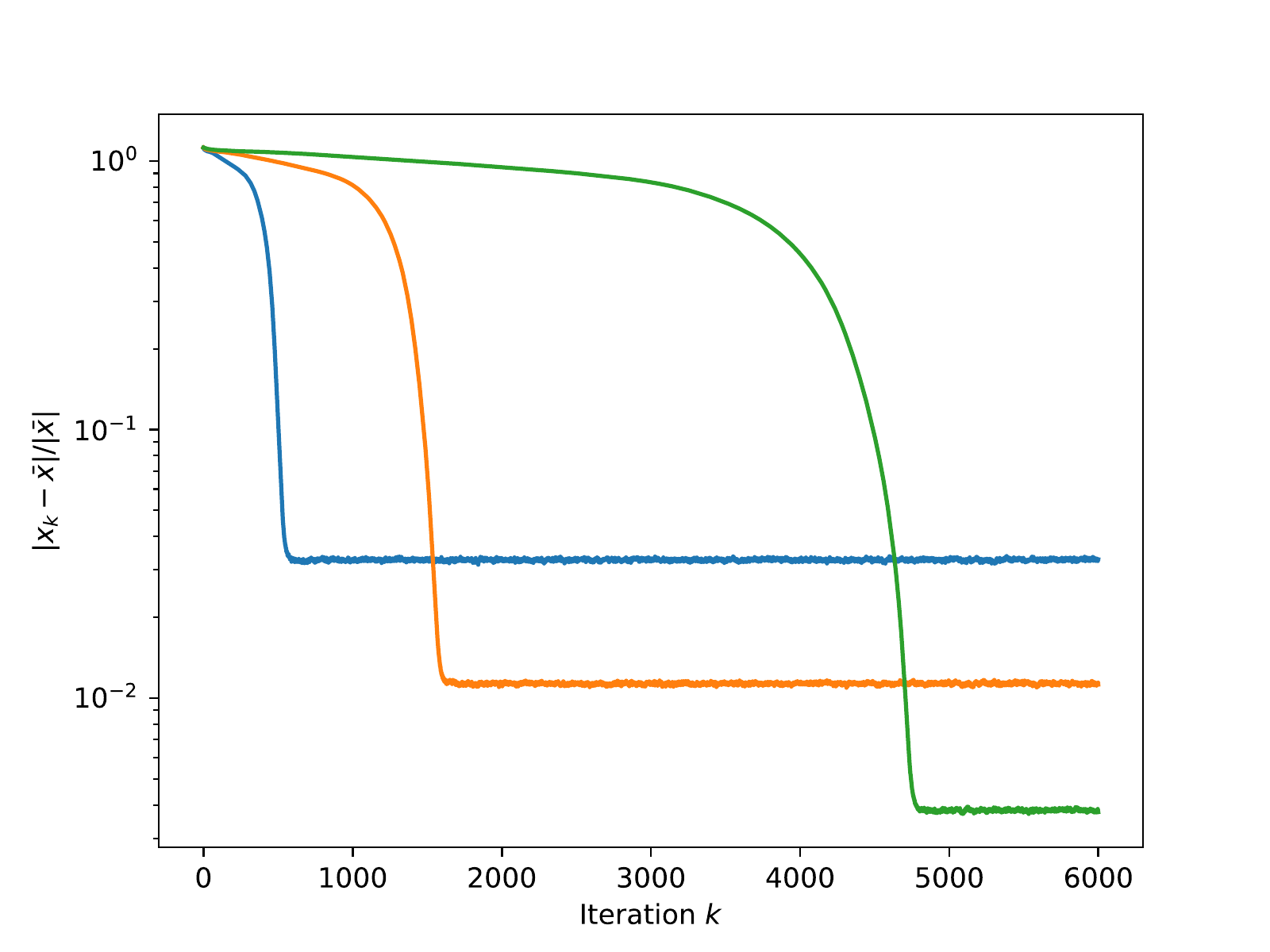}
				\includegraphics[scale=0.56]{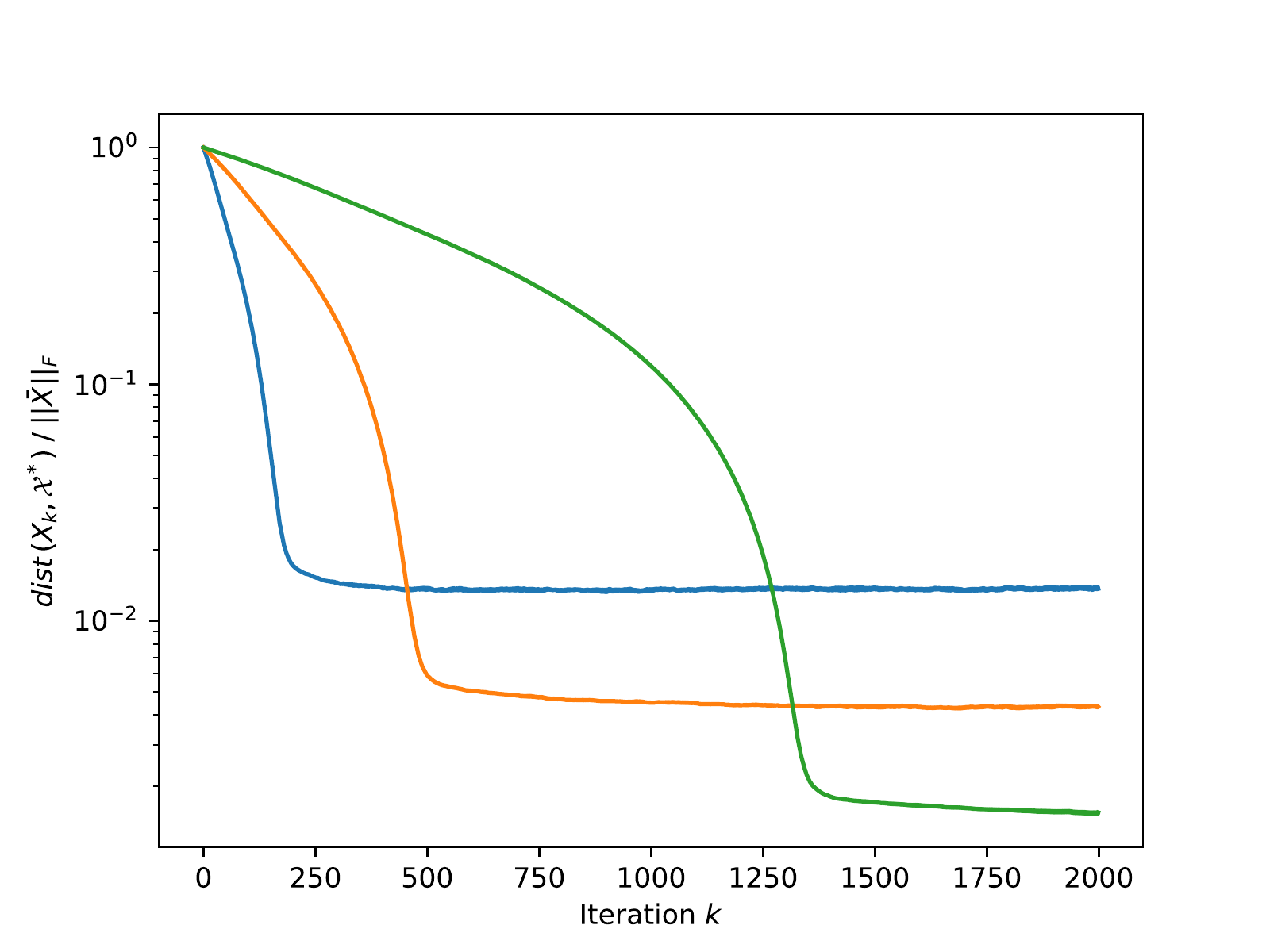}
				\caption{Constant step subgradient method. (Left) { Phase retrieval}  with the corrupted set-up; $d=1000$, $m=3000$, and $\alpha\in\{1,1/3,1/9\}$. 
					(Right) { Covariance matrix estimation} with the corrupted set-up; $d=1000$, $r=3$, $m=10000$,  and $\alpha\in\{1,1/3,1/9\}$. The lower curves curves correspond to smaller step-size in both experiments.
					}\label{fig:const_step}
			\end{figure}

		\begin{thm}[Convergence of fixed stepsize subgradient
		method] \label{thm: bounded_fixed_subgrad} 
Fix a real $\gamma\in (0,1)$ and a real $\alpha> 0$ satisfying 
		\begin{equation}\label{eqn:main_ineq_alph1}
		0 < \alpha < \frac{\gamma\tau}{\sqrt{1+2\tau^2}}\cdot \frac{\mu}{\rho}.
		\end{equation}
		 Let $x_k$ be the iterates generated by Algorithm~\ref{alg:const_step} with stepsize $\alpha$ and initial point $x_0\in \mathcal{T}_{\gamma}$.
		Define the constants 
		$$
		E^* := \left ( \frac{\alpha L}{\mu +
			\sqrt{\mu^2-\alpha \rho L}}  \right )^2\qquad \textrm{and}\qquad D := \sqrt{ \max \{E_0, 2 \alpha^2+ E^*\}}.$$
		Then for each index $k$, the estimates hold:
	$$\sqrt{E_k}\leq D\leq \tfrac{\gamma\mu}{\rho}\qquad\textrm{and}\qquad E_{k}-E^*\leq \max\left\{q^k (E_0-E^*),2
			\alpha^2 \right\},$$	
			where the coefficient $q := 1+\tfrac{\alpha}{L} \left ( \rho-\frac{\mu}{D} \right )$ satisfies $q\in (0,1)$.
		
		%
		%
	\end{thm}
	
	\begin{proof}
	We first verify the claims that are independent of the iteration counter. To this end, observe that \eqref{eqn:main_ineq_alph1} directly implies $\alpha\leq \tau\cdot\tfrac{\mu}{\rho}$, and therefore $E^*$ is well defined.
Next, we show $D< \frac{\gamma\mu}{\rho}$. Indeed, noting $\sqrt{E^*}\leq \alpha\tau^{-1}$ and using \eqref{eqn:main_ineq_alph1}, we deduce 
		$$D^2=\max\{E_0,2\alpha^2+E^*\}\leq \max\{E_0,\alpha^2(2+\tau^{-2})\}\leq \left(\frac{\gamma\mu}{\rho}\right)^2.$$
	Next, we show the inclusion $q\in (0,1)$. To this end, observe 
		$$-1\leq -\frac{\tau\alpha}{D}< \tfrac{\alpha}{L}\left(\rho-\tfrac{ \mu}{D}\right)\leq (1-\gamma^{-1})\cdot\frac{\rho\alpha}{L}<0,$$
		where the first inequality follows from the inequality, $\alpha\leq D$,
 and the third follows from  the inequality, $D<\frac{\gamma\mu}{\rho}$. Thus we conclude $q\in (0,1)$, as claimed.

	We now proceed by induction. Fix an index $k$ and suppose as inductive hypothesis that for each index $i=0,1,\ldots, k$, the estimates hold: 
	$$\sqrt{E_i}\leq D\qquad \textrm{and}\qquad E_{i}-E^*\leq \max\left\{q^{i} (E_0-E^*),2\alpha^2\right\}.$$	
	 Let us consider two cases. Suppose first $E_k\geq E^*$. Then by applying Lemma~\ref{lem:contract}, we deduce
			\begin{align*}
		E_{k+1}-E^*
		 &\leq \left(1+\tfrac{\alpha}{L}\left(\rho-\tfrac{2 \mu}{\sqrt{E_k}+\sqrt{E^*}}\right)\right)(E_k-E^*)\\
		 &\leq \left(1+\tfrac{\alpha}{L}\left(\rho-\tfrac{\mu}{D}\right)\right)(E_k-E^*)\\
		 &=q (E_k-E^*).
		\end{align*}
		Suppose now that the second case, $E_k<E^*$, holds. Then Lemma~\ref{lem:contract} implies 
		\begin{align*}
		E_{k+1}&\leq E_k+\tfrac{\alpha\rho}{L}(E_k-E^*)-2\alpha \tau(\sqrt{E_k}-\sqrt{E^*})\\
		&\leq \max_{E\in [0,E^*]}\{E+\tfrac{\alpha\rho}{L}(E-E^*)-2\alpha \tau(\sqrt{E}-\sqrt{E^*})\}\\
		&=\max\{E^*,\tfrac{\alpha}{L}(2\mu\sqrt{E^*}-\rho E^*)\}.
		\end{align*}
		Subtracting $E^*$, we conclude 
		\[E_{k+1}-E^* \le \max \{0, 2 \tau \alpha \sqrt{E^*}\} \le 2 \alpha^2.\]
		Thus in both cases, we have the estimate 
		$$E_{k+1}-E^*\leq \max \left \{q(E_k-E^*), 2
		\alpha^2  \right \}.$$
		In particular, we immediately deduce  $\sqrt{E_{k+1}} \le D$. Applying the inductive hypothesis, we conclude
		$$E_{k+1}-E^*\leq \max \left \{q\cdot\max\{ q^k(E_0-E^*),2\alpha^2\}, 2
		\alpha^2  \right \}\leq \max\{q^{k+1}(E_0-E^*),2\alpha^2\}.$$ The theorem is proved. 
	\end{proof}

\section{Geometrically decaying step}
	
	In the last section, we showed linear convergence of the constant step size scheme up to a fixed tolerance $E^\ast$. To obtain a linearly convergent method to the true solution set, we will allow the step-size to decrease geometrically. The analogous strategy in the convex setting goes back to \cite{goff}, and our argument follows the same strategy.
	The intuition for why one may expect linear convergence under such step sizes may be gleaned from the Polyak method under the optimal step size 
	$$\alpha_k = \frac{g(x_k) - \min g(x)}{\|\zeta_k\|}.$$
	It is easy to verify that since $E_k$ tend to zero $Q$-linearly, the steps $\alpha_k$ tends to zero R-linearly.
	We implement such a geometrically decaying stepsize in Algorithm~\ref{alg:geometrically_step} and prove linear convergence of the method in Theorem~\ref{thm:geometric}.
	
	\smallskip
	
		\begin{algorithm}[H]
		\KwData{Real $\lambda>0$ and $q\in (0,1)$.}
		

		{\bf Step $k$:}  $(k\geq 0)$ \\
		\qquad Choose $\zeta_k \in \partial g(x_k)$. {\bf If} $\zeta_k=0$, then exit algorithm.
		
		\qquad Set stepsize $\alpha_{k} = \lambda\cdot q^{k}$.\\
		\qquad Update iterate  $x_{k+1}=\proj_{\X}\left(x_{k} - \alpha_k \frac{\zeta_k}{\norm{\zeta_k}}\right)$.\\
		
		
		\caption{Subgradient method with
			geometrically decreasing stepsize }
		\label{alg:geometrically_step}
	\end{algorithm}

	
	\begin{thm} \label{thm:geometric} Fix a real  $\gamma \in (0,1)$  and suppose  $\tau  \le
		\sqrt{ \frac{1}{2-\gamma} }$. 
		Set 
		$$\lambda:=\frac{\gamma \mu^2}{\rho L}\quad \textrm{ and }\quad q:=\sqrt{1-(1-\gamma) \tau^2}.$$ 
		Then the iterates $x_k$ generated by
                Algorithm~\ref{alg:geometrically_step}, initialized at some
                point $x_0 \in \mathcal{T}_{\gamma}$, satisfy:
		\begin{equation} \label{eq:geometric_rate}
		\dist^2(x_k;\X^*) \leq \frac{\gamma^2 \mu^2}{\rho^2}
		\left(1-(1-\gamma)\tau^2\right)^{k}.
		\end{equation}
	\end{thm}
	
	\begin{proof} 
		We will prove the result by induction. 
		To this end, suppose the bound \eqref{eq:geometric_rate}  holds for all
		$i = 0, \hdots, k$. Appealing to Lemma~\ref{lem:key_rec}.
and using the relation $\alpha_k=\lambda q^k$, we obtain
		\begin{equation}\label{eqn:key_eqn2}
		E_{k+1}\leq \left (1 + \frac{\rho\lambda q^k}{L}
                \right ) E_k-2\lambda \tau  q^k \sqrt{E_k}+\lambda^2 q^{2k}.
		\end{equation}
		Define the constant $M := \frac{\gamma \mu}{\rho}$. Recall the induction assumption guarantees
		$\sqrt{E_k} \leq Mq^k$. Let us therefore fix some value $R\in [0,M]$ satisfying $\sqrt{E_k} = R
		q^k$.
		Inequality \eqref{eqn:key_eqn2} then implies
		$$E_{k+1}\leq \max_{R\in [0,M]} \left
                  \{R^2q^{2k}+\frac{\rho\lambda
                    R^2}{L}q^{3k}-2\lambda \tau R q^{2k}+ \lambda^2
                  q^{2k} \right \}.$$
		Note that the expression inside the maximum is a convex quadratic in $R$ and therefore 
		 the maximum must occur either at $R=0$ or $R=M$. We therefore deduce 
		\begin{equation} \label{eq:recurrence_relation}
                  E_{k+1}\leq q^{2k}\cdot \max\left \{\lambda^2 ,~
		M^2+\frac{\rho\lambda}{L} M^2q^{k}-2\lambda \tau M
                + \lambda^2 \right \}.
		\end{equation}
	To complete the induction, it is therefore sufficient to show
		\begin{equation} \label{eq:M_condition} \lambda^2  \leq M^2 q^2 \qquad \textrm{ and }\qquad
			M^2+\frac{\rho\lambda}{L} M^2q^{k}-2\lambda \tau M + \lambda^2  \leq
			M^2q^2. \end{equation}
		First, we show that $M$ satisfies the first property. Note the equality $M =
		\frac{\lambda }{\tau}$. Hence, it suffices to show that $\tau \le q$,
Observe that the assumption $\tau \le \sqrt{\frac{1}{2-\gamma}}$ directly implies
		\begin{align*}
		\tau^2 +
		(1-\gamma) \tau^2 \le 1. 
		\end{align*}
		Rearranging yields $\tau^2 \le
		1-(1-\gamma)\tau^2 = q^2$. Hence, the first
		condition in \eqref{eq:M_condition} holds.
		
		Next we show that $M$ satisfies the second property in
		\eqref{eq:M_condition}. Thus, rearranging the expression, we must establish
		\begin{equation} \label{eq:M_quad_condition}
		\left (1+\frac{\rho\lambda}{L} q^{k}-q^2 \right
                )M^2-2\lambda \tau M + \lambda^2  \leq
		0.\end{equation}
		We will show that the quadratic on the left-hand-side in $M$ has two real positive roots. To this end, a quick computation shows that the two roots are 
		\begin{align*}\frac{\lambda
			\tau \pm\sqrt{\lambda^2\tau^2-\lambda^2\left
                  (1+\frac{\rho\lambda}{L}
				q^{k}-q^2\right )}}{1+\frac{\rho\lambda}{L} q^{k}-q^2}
		&= \frac{\lambda}{\tau \mp \sqrt{\tau^2 - \left (1+
                  \frac{\rho \lambda}{L} q^k - q^2\right )}}.
		\end{align*}
	To see that the discriminant is nonnegative, observe
		\begin{align*}
		\tau^2-\left (1+ \frac{\rho \lambda}{L} q^k-q^2
                  \right ) &\ge \tau^2 - \left (1+\frac{\rho \lambda}{L}
		-q^2 \right )
		= \tau^2 -(1 + \gamma \tau^2 -q^2) =0. 
		\end{align*}
		Thus the convex quadratic in \eqref{eq:M_quad_condition} has two real roots, and our choice $M=\lambda/\tau$ lies between them.
		Hence the condition \eqref{eq:M_quad_condition} holds, and the inductive step is complete.

	\end{proof}
	
	We now illustrate the performance of Algorithm~\ref{alg:geometrically_step} on our two running examples in Figure~\ref{fig:geo_step}. 
Empirically, we observed that $\lambda> 0$ and $q\in (0,1)$ must be tuned for performance, which is what we did in the experiments.
	We observe linear convergence in all cases, and the convergence rate of the method improves monotonically as the chosen rate $q$ is decreased. While the Polyak scheme pictured in Figure~\ref{fig:pol_phase_cov} clearly outperforms all other methods, the geometrically decaying step scheme performs much better than the constant step scheme in Figure~\ref{fig:const_step}. 
		\begin{figure}[!h]
			\centering
			\includegraphics[scale=0.56]{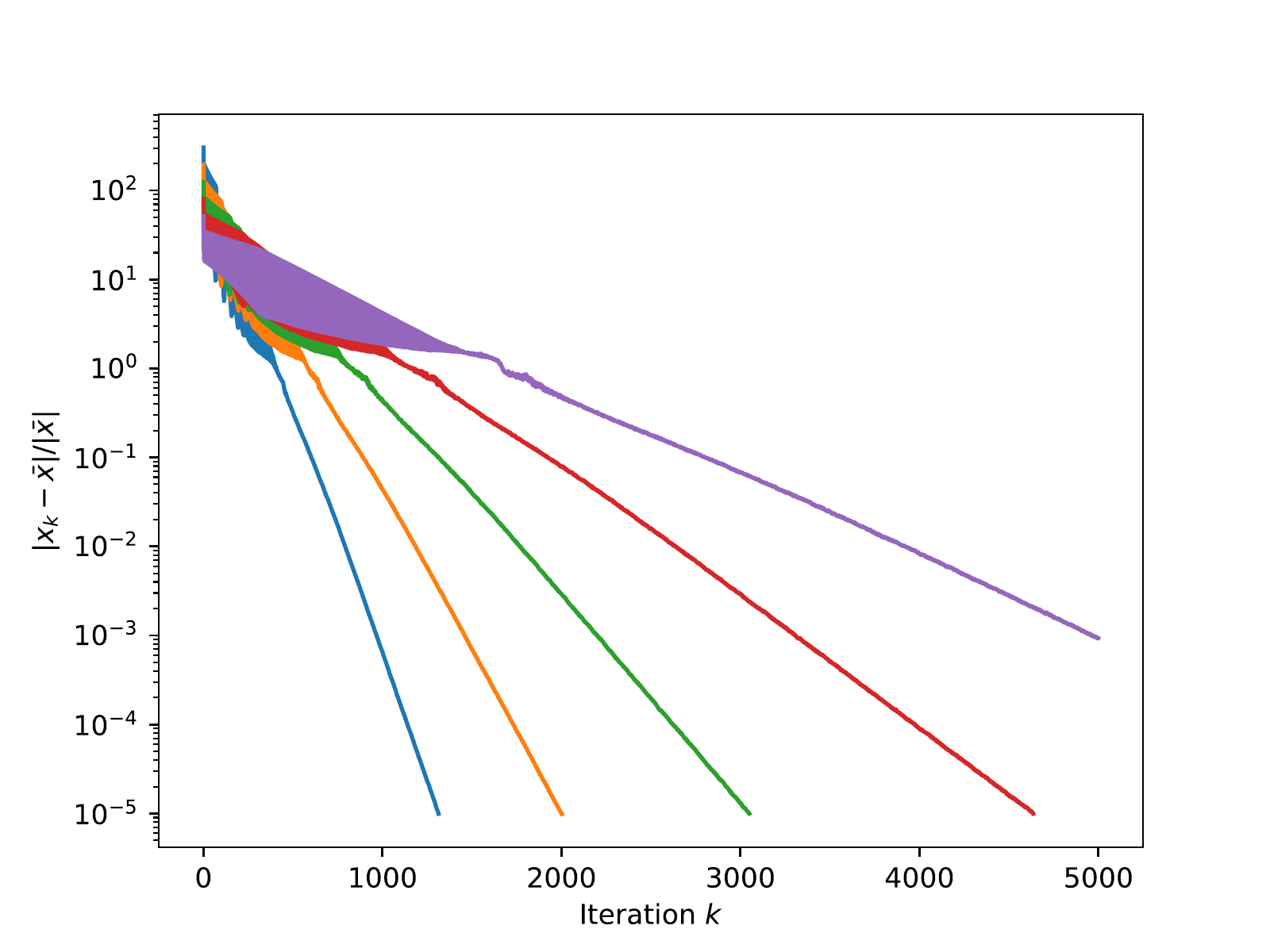}
			\includegraphics[scale=0.56]{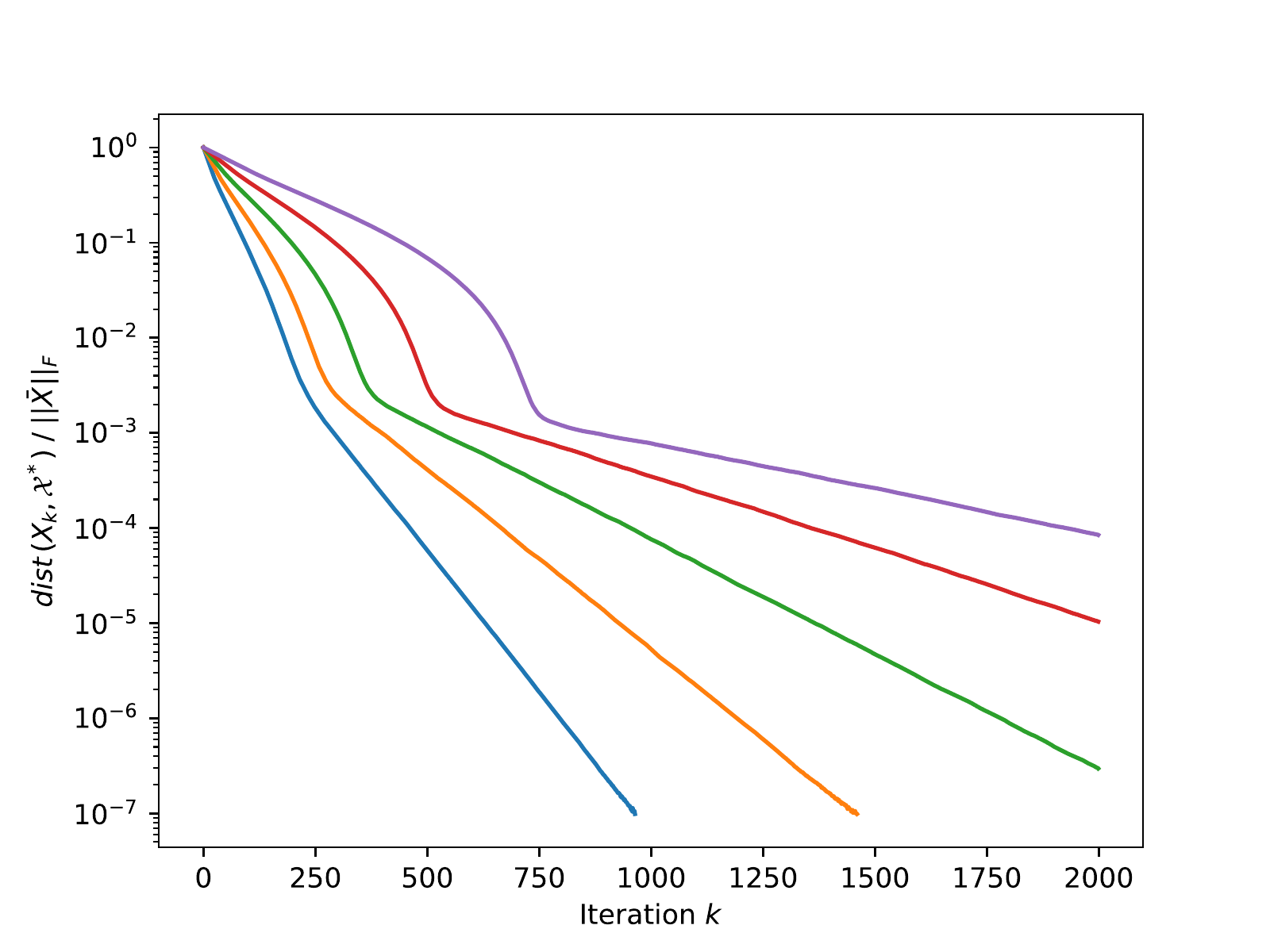}
			\caption{Geometrically decaying step size. (Left) { Phase retrieval}  with the corrupted set-up; $d=1000$, $m=3000$, $q\in\{0.983,0.989,0.993,0.996,0.997\}$. (Right) { Covariance matrix estimation} with the corrupted set-up; $d=1000$, $r=3$, $m=10000$, 
			$q\in\{0.986, 0.991, 0.994, 0.996, 0.998\}$. The depicted rates uniformly improve with lower values of $q$, in both figures.}
			\label{fig:geo_step}
		\end{figure}

\let\oldbibliography\thebibliography
\renewcommand{\thebibliography}[1]{%
  \oldbibliography{#1}%
  \setlength{\itemsep}{-1pt}%
}

	\bibliographystyle{plain}
	\bibliography{bibliography}

\end{document}